\documentclass[11pt]{article}

\usepackage{amsmath,amssymb,amsthm,mathtools,mathrsfs,enumitem,comment,booktabs,hyperref}

\title{Failure of Strong Convergence of Matrices with Fermionic Entries}
\author{Dimitri Shlyakhtenko}
\date{\today}

\newtheorem{theorem}{Theorem}
\newtheorem{lemma}{Lemma}
\newtheorem{corollary}{Corollary}

\newtheorem{proposition}{Proposition}

\newcommand{\Cl}{\operatorname{Cl}}
\newcommand{\atimes}{\dot{\otimes}}
\newcommand{\myd}{d}
\begin{document}
\maketitle
\begin{abstract}
    Let $Q^{(k)}_N$ be an $N\times N$ matrices with entries satisfying CAR, normalized to have variance $1/\sqrt{N}$ with respect to the trace of the CAR algebra. We show that, although the operator norm of the real part of an individual matrix $Q^{(k)}_N$ converges as $N\to\infty$ to the semicircular limit, the family of matrices does not converge to the free probability limit strongly. In fact, even the operator space structure of the linear spans of the real and imaginary parts of $Q^{(k)}_N$'s, $k=1,\dots,M$, does not converge to the semicircular limit. 
\end{abstract}
\section{Introduction.}
In this paper, we consider the asymptotics of $N\times N$ block matrices $A^{(k)}_N$, $1\leq k\leq M$ with entries given by generators satisfing Canonical Anti-Commutation Relations (CAR) (see \S\ref{sect:CARMatrices} for the precise definitions and notations).  These matrices are fermionic analogs of GUE matrices, in the sense that the Gaussian random entries of GUE matrices are replaced by ``iid'' fermionic variables.  It has been known, since the work of Voiculescu, that, under a natural trace, such ``CAR-UE'' matrices converge in moments to a free semicircular family $S_1,\dots,S_M$ (see \cite{Voiculescu:Random}).  However, two questions were left open after Voiculescu's work. The first is the behavior of the operator norm of a single matrix $A_N=A_N^{(1)}$.  The second is the question of strong convergence: whether it is true that operator norms of arbitrary polynomials in the matrices $A_N^{(1)},A_N^{(2)},\dots$ converge to the operator norm of the same polynomial but evaluated in the free semicircular family $S_1,S_2,\dots$.  This question was communicated to the author by Gilles Pisier, and it came up, for example, in the work \cite{Pisier-Shlyakhtenko:Grothendieck} where the matrices $A_N^{(k)}$ were used as one of the ways to prove the QWEP property for $C^*$-algebras generated by generalized circular elements. 

Our main result is that for all $N$,  $\Vert A_N \Vert = \sqrt{4-3/N} < 2$ (in fact, we give a precise description of the spectrum of $A_N$, see Theorem~\ref{thrm:spectrumOfA}); moreover, $\Vert A_N\Vert\to 2$, which is the operator norm of the semicircular limit. However, we show that strong convergence fails if we consider simultaneously $12$ or more such matrices. More precisely, if $A_N^{(1)}, B_N^{(2)},\dots,A_N^{(M)},B_N^{(M)}$ are ``CAR-iid'' copies of $A_N^{(1)}$, and we set $Q_N^{(j)} = A_N^{(j)} + i B_N^{(j)}$, then $Q_N^{(j)}$ converge in $*$-moments to a family of freely independent circular elements $Q^{(1)},\dots,Q^{(M)}$, but $$\limsup_{N\to\infty} \Big\Vert \sum_{j=1}^M Q_N^{(j)} (Q_N^{(j)})^*\Big\Vert \gneq \Big\Vert \sum_{j=1}^M Q^{(j)} (Q^{(j)})^*\Big\Vert. $$ (see Theorem~\ref{thrm:noStrong} below). The same example shows that the operator space structure of the span of $A_N^{(1)}, B_N^{(2)},\dots,A_N^{(M)},B_N^{(M)}$ does not converge, as $N\to\infty$ to the free probability limit. Thus the asymptotics of the ``CAR-UE'' matrices gives another example of a real Hilbertian operator space \cite{Pisier:Book}, which is not isomorphic to the complex span of semicircular variables, and whose precise operator structure we cannot elucidate.  Convergence to the semicircular limit does hold for real-linear combinations, but we are unable to prove the complex case (although we suspect that the statement is true). We summarize these results in Table~\ref{tbl:asymptoticComparison}.

\begin{table}[h]
\begin{center}
\begin{tabular}{lll}
Space & Comparison to Free  &  Isomorphism Class \\
\toprule
Banach space $W_\infty^\mathbb{R}$ &  same & $ \ell_\mathbb{R}^2(2M)$  \\
Banach space $W_\infty^\mathbb{C}$ &
unknown, same? &  conjecturally $\ell^2_{\mathbb{C}}(2M)$ \\
Operator space  $W_\infty^\mathbb{C}$  & 
different & unknown \\ 
$C^*$-algebra $C^*(W_\infty^{\mathbb{C}})$ & different & unknown \\
\bottomrule 
\end{tabular}
\caption{Here $W^{\mathbb{C}}_\infty$ and $W^{\mathbb{R}}_\infty$ are the complex or real span of the family $(A^{(j)}_N,B^{(j)}_N, 1\leq j\leq M)$ in the ultraproduct of  $W^*(A^{(j)}_N,B^{(j)}_N:1\leq j\leq M).$ } \label{tbl:asymptoticComparison}
\end{center}
\end{table}

There is another obstruction to strong convergence, where a different degree-two polynomial involving commutators of the ``CAR-UE'' matrices is used.  This computation is given in Theorem~\ref{thrm:noStrongWeaker}. Based on this norm computation, we give evidence in  \S\ref{sec:OtherOmeags} that there is not a ``a single-vector obstruction to norm convergence'' as is the case for permutation matrices, and showing that it does not seem possible to ``restore'' strong convergence by restricting to a ``large enough'' part of the underlying Hilbert space.

To conclude the paper, we discuss a variant of the ``CAR-UE'' matrices involving $q$-creation operators \cite{Bozejko-Speicher:qBrownian} for $-1<q<1$. We show that in those cases, strong convergence holds true, in fact not just for the ``$q$-UE'' matrices, but for matrices with creation operators as entries (and free creation operators as limits).

\subsection{AI use: How the results were obtained}
This paper also represents an experiment in the use of AI for research in operator algebras. The counterexample to strong convergence was found by utilizing software written by ChatGPT 5.4 Pro, and ChatGPT was able to find an expression for the special vector $\eta_0$ used to give a lower bound on the norm of sums of commutators in Lemma~\ref{lemma:normOfCommutator}. However, it did not realize that this gives a counterexample to strong convergence.  ChatGPT, however, showed remarkable performance on the question of convergence of $A_N$ (though when asked directly it hallucinated incorrect operator-space inequalities). It hallucinated the combinatorial identity \eqref{eq:MomentsAreSuperBallots} discussed in  \S\ref{subsec:SpectralMeasureAN} for the moments of the matrices $A_N$. When explained that this identity does exist in the paper \cite{gessel} and asked to give a direct proof, it retraced the steps back to the matrix $A_N$, found the representation theory proof presented in this paper, and derived the identity as a consequence. ChatGPT 5.5 Pro was also helpful in the proof of the $q$-deformed version. Overall, it is clear that it is a useful and powerful tool for mathematical research, even in a highly theoretical field such as the field of this paper. 

A partial formalization of the results of this paper in Lean is available at \url{https://github.com/shlyakhtenko/car-matrices}. The formalization covers \S2, including the underlying representation-theory tools, as well as Lemma~\ref{lemma:normOfCommutator}, down to theorems and definitions in Mathlib v4.32.0-rc1 with no additional axioms or hypotheses. We intend to pursue formalization of the free probability norm estimates necessary to see the violation of norm convergence later on.

\subsection{Setup and Notation}\label{sect:CARMatrices}
For fixed integers $M$ and $N$, consider the real Hilbert space  \[H=\bigoplus_{k=1}^M \bigl(H_{k}^{(1)} \oplus H_{k}^{(2)}\bigr),\] where for each 
$k=1,2,\dots,M$ and $r=1,2$,  $H_k^{(r)}$ is  a real Hilbert space of dimension $N^2$ with basis $h_{pq}^{(r,k)}$, $1\leq p,q\leq N$.  Let $\mathscr{F}$ denote the associated antisymmetric Fock space (see e.g. \cite{bratteli-robinson}), and let $\Omega$ be the vacuum vector. Denote by $\mathbb{B}$ the space of (bounded) operators on the finite-dimensional space $\mathscr{F}$. 

For $h \in H\otimes_{\mathbb{R}} \mathbb{C}$, let $a(h)^*\in \mathbb{B}$ be the creation operator given on elementary antisymmetric tensors by $ a(h)^* (h_1\wedge \dots \wedge h_k) = h\wedge h_1\wedge \dots \wedge h_k$,
and let $a(h)$ be its adjoint.  These operators satisfy the Canonical Anti-Commutation Relations: for $h,g\in H\otimes_{\mathbb{R}} \mathbb{C}$,
\[
a(h) a(g) + a(g) a(h) = 0,\ a(h)a(g)^* + a(g)^* a(h) = \langle h,g\rangle.
\]

For $h\in H$, let
\[
b(h) = a(h) + a(h)^*
\] be the associated Majorana operator.

Let $e_{pq} \in M_{N\times N}$, $1\leq p,q\leq N$, be elementary $N\times N$ matrices, and consider the operators
\[
Q^{(k)}_N = \frac{1}{\sqrt{N}} \sum_{p,q=1}^N e_{pq} \otimes \bigl(b(h^{(1,k)}_{pq}) + i\, b(h^{(2,k)}_{pq})\bigr) \in M_{N\times N}\otimes \mathbb{B}.
\] 

Let
\[
A^{(k)}_N = \frac{Q^{(k)}_N + (Q^{(k)}_N)^*}{2}, \qquad
B^{(k)}_N = \frac{Q^{(k)}_N - (Q^{(k)}_N)^*}{2i}
\] be the real and imaginary parts of $Q_N^{(k)}$.  Denote by $\phi:\mathbb{B}\to\mathbb{C}$ the vector state $x\mapsto \langle \Omega , x\Omega\rangle$ and by $\phi_N: M_{N\times N}\otimes \mathbb{B}\to \mathbb{C}$ the state $\phi_N= N^{-1} Tr\otimes \phi$.

By \cite{Voiculescu:Random}, the family of matrices  $(A_N^{(1)},B_N^{(1)},\dots,A_N^{(M)}, B_N^{(M)})$ converges in distribution to a free semicircular family.

\section{Norm of real linear combinations.}
In this section we show that operator norms of real linear combinations of words of length $1$ in the matrices $A_N^{(r)}$ and $B_N^{(r)}$ do converge to operator norms of their free probability limit laws. We start with the case of a single matrix.

\subsection{A Clifford algebra picture of the entries of $A_N$.}

To simplify notation, let us set
$$A_N = A_N^{(1)} = \frac{1}{\sqrt{N}}\sum e_{pq} \otimes \myd_{pq},$$
where
\begin{eqnarray*} 
\myd_{pq} &=& \frac{1}{2}(b(h^{(1)}_{pq})  + i b(h^{(2)}_{pq})  + 
b(h^{(1)}_{qp}) - i b(h^{(2)}_{qp})) \\
&=& \frac{1}{2}\bigl( a(h^{(1)}_{pq}) + a(h^{(1)}_{pq})^* + a(i h^{(2)}_{pq}) +
a(-i h^{(2)}_{pq})^* \\ && + a(h^{(1)}_{qp}) + a(h^{(1)}_{qp})^* + 
a(-i h^{(2)}_{qp}) + a(i h^{(2)}_{qp})^*\bigr) \\
&=& a\bigl(\frac{1}{2}(h^{(1)}_{pq} + i h^{(2)}_{pq} + h^{(1)}_{qp} -i h^{(2)}_{qp})\bigr) 
\\ && + a\bigl(\frac{1}{2}( h^{(1)}_{pq} -i h^{(2)}_{pq} + h^{(1)}_{qp} + i h^{(2)}_{qp})\bigr)^* \\
&=& a(g_{pq}) + a(g_{qp})^*
\end{eqnarray*}
where $g_{pq} =\frac{1}{2}(h^{(1)}_{pq} + i h^{(2)}_{pq} + h^{(1)}_{qp} -i h^{(2)}_{qp}) $ are orthonormal vectors.

Note that $\myd_{pq}=\myd_{qp}^*$; using CAR one easily sees that $\myd_{pq}$ satisfy the Clifford relations:
$$
\myd_{pq}\myd_{q'p'}+\myd_{q'p'}\myd_{pq}=2\delta_{pp'}\delta_{qq'}.
$$

Thus the algebra generated by the entries $\myd_{pq}$ is isomorphic to the Clifford algebra $\Cl(W,B)$, where $W=M_{N\times N}(\mathbb{C})$ and $B$ is the bilinear form $B(x,y)=Tr(xy)$. 

If we denote by $S$ the underlying complex vector space of $\Cl(W,B)$, then $\dim S=2^{N^2}$. The state $\phi$ then corresponds to the normalized trace on $\textrm{End}(S)$ restricted to $\Cl(W,B)$ acting by left multiplication. 

\subsection{The \(\mathfrak{gl}_N\) module structure of the Clifford algebra}

The left multiplication action of $\Cl(W,B)$ gives rise to a natural $\mathfrak{gl}_N$ action on $S$ which will be useful to understand the action of the matrix $A_N$.  We introduce this action below. The main result of this section is Proposition~\ref{prop:decompositionOfS}, which identifies $S$ as a multiple of a single irreducible highest weight representation. We could not find this exact result in the literature; for similar results, see for example \cite[Proposition 2.4(ii), Example 2.5(1)]{panyushev} and \cite[pp. 356-358]{kostant}. We thus give a direct argument.

For \(1\le a,b\le N\), let
\[
        F_{ab}=\frac12\sum_{k=1}^N\myd_{ak}\myd_{kb}.
\]

\begin{lemma} \label{lem:FgivesAdjointAction}
For all \(a,b,p,q\),
\[
        [F_{ab},\myd_{pq}]
        =
        \delta_{bp}\myd_{aq}-\delta_{aq}\myd_{pb}.
\]
\end{lemma}

\begin{proof}
Using the Clifford relations, we have:
\begin{eqnarray*}
    [F_{ab},\myd_{pq}] &=& \frac{1}{2}\sum_{k=1}^N 
    [\myd_{ak}\myd_{kb},\myd_{pq}] \\
    &=& \frac{1}{2}\sum_{k=1}^N \myd_{ak} \myd_{kb}\myd_{pq} -  \myd_{pq} \myd_{ak}\myd_{kb} \\
    &=& \frac{1}{2}\sum_{k=1}^N   2\delta_{pb}\delta_{qk}\myd_{ak}  - \myd_{ak} \myd_{pq} \myd_{kb} - 2\delta_{pk}\delta_{qa}\myd_{kb} + \myd_{ak} \myd_{pq}\myd_{kb} \\
    &=& \delta_{bp}\myd_{aq} - \delta_{aq}\myd_{pb},
\end{eqnarray*} as claimed.
\end{proof}

\begin{lemma} \label{lem:FareGLN}
The operators \(F_{ab}\) satisfy:
\[
        [F_{ab},F_{cd}]
        =
        \delta_{bc}F_{ad}-\delta_{ad}F_{cb}.
\]

Thus if $E_{ij}$ are the generators of \(\mathfrak{gl}_N\) corresponding to elementary matrices with a single nonzero entry in the $i,j$-th coordinate, then the map  $\pi_S: E_{ab}\mapsto F_{ab}$ (with $F_{ab}$  acting on $S$ by left multiplication) is a Lie algebra representation of  \(\mathfrak{gl}_N\) on $S$.
\end{lemma}
\begin{proof}
We compute, using the previous Lemma and the fact that $[F_{ab},\cdot]$ is a derivation:
\begin{eqnarray*}
     [F_{ab},F_{cd}] &=&  \frac12\sum_{k=1}^N [F_{ab},\myd_{ck}\myd_{kd}] 
        = \frac12\sum_{k=1}^N [F_{ab},\myd_{ck}]\myd_{kd} +
            \myd_{ck}[F_{ab},\myd_{kd}] \\
        &=&  \frac12\sum_{k=1}^N \delta_{bc}\myd_{ak}\myd_{kd} - \delta_{ak}\myd_{cb}\myd_{kd} + \myd_{ck} \delta_{bk} \myd_{ad} 
            -\myd_{ck}\delta_{ad} \myd_{kb} \\
        &=& \delta_{bc} F_{ad} - \frac12 \myd_{cb}\myd_{ad} + \frac12 \myd_{cb}\myd_{ad} - \delta_{ad}F_{cb} \\
        &=& \delta_{bc} F_{ad}  - \delta_{ad}F_{cb},
\end{eqnarray*} as claimed.  These are exactly the relations between the generators $E_{ab}$ of $\mathfrak{gl}_N$ and thus the map $E_{ab}\mapsto F_{ab}$ is a Lie algebra homomorphism. 
\end{proof}

\subsection{The highest weight vectors $h_{I,\eta}$}

Given a subset $I\subset\{1,\dots,N\}$,  define 
    $$d_I = \prod_{p\in I} \myd_{pp}$$ (ordered product taken in increasing order).   

    Fix once and for all an order on the set $\{(p,q):p<q\}$. Given a function $\eta:\{(p,q):p<q\}\to\{0,1\}$, let $$
    y^\eta_{pq} = \begin{cases*} \frac{1}{2} \myd_{pq}\myd_{qp},& if $\eta(p,q)=0$\\ \myd_{pq},& otherwise.\end{cases*}$$ Finally, let $$
        h_{I,\eta} =  d_I \prod_{p<q} y_{pq}^\eta 
    $$ (product taken in the fixed order on the index set), and consider the submodule $S_{I,\eta}$ of $S$ generated by $h_{I,\eta}$.

    For each $I$ and $\eta$, define \[
    r_{I,\eta}
    =
    d_I
    \prod_{\substack{i<j\\ \eta(i,j)=1}}
    s_{ij},\qquad\text{where }  s_{ab} = \myd_{ab}+\myd_{ba}. 
    \]

\begin{lemma}\label{lem:rIeta}
        Let $h_0 = h_{\emptyset, 0} = \prod_{p<q}\frac{1}{2}\myd_{pq}\myd_{qp}$.
    Then: \begin{enumerate}[label={\rm(\roman*)}] \item $r_{I,\eta}$ is invertible  \item $h_0\,r_{I,\eta}=h_{I,\eta}.$ \end{enumerate}
\end{lemma}
\begin{proof}
    If $a<b$, then $s_{ab}^2 = \myd_{ab} \myd_{ab} + \myd_{ab}\myd_{ba}+\myd_{ba}\myd_{ab} + \myd_{ba}\myd_{ba} = 2$ so that $s_{ab}$ is invertible and $s_{ab}^{-1}= \frac{1}{2} s_{ab}$.  Furthermore, $\myd_{pp} \myd_{pp}=1$, so $\myd_{pp}$ is also invertible. Thus $r_{I,\eta}$ is a product of invertible elements.  This gives (i).

    For (ii), we note that if $p<q$, $a<b$, and $(p,q)\neq(a,b)$, then  the even element
    $\myd_{pq}\myd_{qp}$ commutes with $s_{ab}  = \myd_{ab}+\myd_{ba}$, and thus $$\bigl(\frac{1}{2}\myd_{pq}\myd_{qp} \bigr)s_{ab}=s_{ab} \bigl(\frac{1}{2}\myd_{pq}\myd_{qp}\bigr).$$
    If on the other hand $(p,q)=(a,b)$, then
    \begin{eqnarray*}
        \bigl(\frac{1}{2}\myd_{pq}\myd_{qp}\bigr) s_{pq} &=& \frac{1}{2}(\myd_{pq}\myd_{qp}\myd_{pq} + \myd_{pq}\myd_{qp} \myd_{qp}) \\
        &=& \frac{1}{2} \myd_{pq} (2-\myd_{pq}\myd_{qp}) + \frac{1}{2} \myd_{pq}\myd_{qp} \myd_{qp} \\&=& \myd_{pq} - \frac{1}{2} \myd_{pq}\myd_{pq} \myd_{qp} + \frac{1}{2}\myd_{pq}\myd_{qp}\myd_{qp} \\ &=& \myd_{pq}.
    \end{eqnarray*}
    Similarly, if $p<q$ and $a$ is arbitrary, then $$\bigl(\frac{1}{2}\myd_{pq}\myd_{qp}) \myd_{aa}
    =\myd_{aa} \bigl(\frac{1}{2} \myd_{pq} \myd_{qp}) $$ because  $(p,q)\neq(a,a)$, $(q,p)\neq (a,a)$, and thus the even element $\myd_{pq}\myd_{qp}$ commutes with $\myd_{aa}$.  Thus
    \begin{eqnarray*}
        h_0 r_{I,\eta} &=& \prod_{p<q} \frac{1}{2} \myd_{pq}\myd_{qp} \cdot \prod_{a\in I}\myd_{aa} \cdot \prod_{\substack{i<j\\ \eta(i,j)=1}} 
    s_{ij}\\ &=& \prod_{a\in I}\myd_{aa} \cdot \prod_{p<q} \frac{1}{2} \myd_{pq}\myd_{qp} \cdot \prod_{\substack{i<j\\ \eta(i,j)=1}} s_{ij} \\ &=& \prod_{a\in I}\myd_{aa} 
    \prod_{p<q} y_{pq}^\eta = h_{I,\eta}.
    \end{eqnarray*}
\end{proof}

\begin{lemma}
    Let $S_{I,\eta}$ be the $\mathfrak{gl}_N$-submodule of $S$ generated by $h_{I,\eta}$. Then $S_{I,\eta}$ is isomorphic as a $\mathfrak{gl}_N$ module to the irreducible module $L_\nu$ with highest weight $\nu = (N-1/2, N-3/2,\dots,3/2,1/2)$ and dimension $2^{N(N-1)/2}$.
\end{lemma}
\begin{proof}
By Lemma~\ref{lem:rIeta}, it's sufficient to prove that $S_{\emptyset,0}$, the submodule generated by $h_0$, is $L_\nu$.  Indeed, the invertible map $$x\mapsto x r_{I,\eta}$$ clearly commutes with the left multiplication action of $\Cl(W,B)$ on itself, and thus with the action of $\mathfrak{gl}_N$; moreover, it carries $h_0$ into $h_{I,\eta}$. 

The action of $\mathfrak{gl}_N$ on $S$ is semisimple, and thus $S$ decomposes into a direct sum of irreducible modules.

We now check that $h_0 =\prod_{p<q} \frac{1}{2}\myd_{pq}\myd_{qp} $ is a highest weight vector (see e.g. \cite[\S3.2.1]{GoodmanWallach}).  Note that if $(p,q)\neq (p',q')$, $p<q$, $p'<q'$, then $\myd_{pq}\myd_{qp}$ and $\myd_{p'q'}\myd_{q'p'}$ commute. Thus we can reorder the terms in the product defining $h_0$ at will.

We first check that the action of the raising operators $F_{ab}$ for $a<b$ is trivial.

Assume that $a<k$ and write $h_0 = \myd_{ak}\myd_{ka} h'$, where $h'$ is (one half times) the product of $\myd_{pq}\myd_{qp}$ for $p<q$ and $(p,q)\neq (a,k)$. 

Then \begin{eqnarray*}
    \myd_{ak}\myd_{kb} h_0 &=&  \myd_{ak}\myd_{kb} \myd_{ak}\myd_{ka} h' \\
    &=&  - \myd_{ak} \myd_{ak} \myd_{kb} \myd_{ka} h' = 0,
\end{eqnarray*}
since $\myd_{ak}^2=0$.

Assume now that $a\geq k$, so $k<b$.  Write $h_0 = \myd_{kb}\myd_{bk} h''$.  Then $$
    \myd_{ak}\myd_{kb} h_0  =  \myd_{ak}\myd_{kb} \myd_{kb}\myd_{bk} h'' = 0,
$$
since   $\myd_{kb} \myd_{kb}=0$.

It follows that for any $a<b$,  $$F_{ab} h_0 = \frac{1}{2}\sum_k \myd_{ak}\myd_{kb} h_0 =0.$$

Next we check that $h_0$ is an eigenvector for the Cartan subalgebra. Suppose that $a<k$.  
Write $h_0 = \myd_{ak}\myd_{ka}h'$.  
\begin{eqnarray*}
    \frac{1}{2}\myd_{ak}\myd_{ka} h_0 &=& \frac{1}{2}\myd_{ak}\myd_{ka}   \myd_{ak}\myd_{ka}h' \\ &=& \frac{1}{2}\myd_{ak} (2 - \myd_{ak}\myd_{ka}) \myd_{ka} h' = h_0.
\end{eqnarray*}
Suppose that $a= k$. Then $$\frac{1}{2} \myd_{aa}\myd_{aa} h_0 
= \frac{1}{2} h_0.$$
Finally, suppose that $a>k$. Write 
$h_0 = \myd_{ka}\myd_{ak} h''$.  Then
$$ \frac{1}{2}\myd_{ak}\myd_{ka} h_0=  \frac{1}{2}\myd_{ak}\myd_{ka} \myd_{ka}\myd_{ak} h''=0.$$
Thus
\begin{eqnarray*}
\pi_S(E_{aa})h_0 &=& \frac{1}{2} \sum_{k:1\leq k\leq N} \myd_{ak} \myd_{ka} h_0 
\\ &=& \sum_{k:a<k\leq N} \frac{1}{2} \myd_{ak} \myd_{ka} h_0  
+ \sum_{k:k=a}\frac{1}{2}\myd_{kk}\myd_{kk} h_0 + \frac{1}{2}\sum_{k:a>k\geq 1} \myd_{ak} \myd_{ka} h_0 \\
&=& \left(\sum_{k:a<k}  h_0\right) + \frac{1}{2}h_0 
= (N-a+\frac{1}{2}) h_0. 
\end{eqnarray*}
Thus if we fix $\lambda_1,\dots,\lambda_N$ and consider the Cartan subalgebra element $$H^\lambda =\sum_{a=1}^N \lambda_a E_{aa}  $$ then $$\pi_S(H^\lambda) h_0 = \sum_{a=1}^N 
\left[\bigl(N-a+\frac{1}{2}\bigr)\lambda_a  \right] h_0.$$ 
It follows that $h_0$ is indeed the highest weight vector with weight $\nu$. 
Since the $\mathfrak{gl}_N$ action is semisimple, the cyclic submodule generated by a highest-weight vector of weight $\nu$ is a direct sum of copies of $L_\nu$; but because it is generated by a single highest-weight vector, it must be a single copy and thus is isomorphic to $L_\nu$.

By the Weyl dimension formula (see e.g., \cite[\S7.1, Theorem 7.1.9]{GoodmanWallach}), $$\dim L_\nu = \prod_{1\leq p<q\leq N} \frac{N-p+\frac{1}{2} - N+q -\frac{1}{2} +q-p }{q-p} =2^{N(N-1)/2}.$$ 
\end{proof}
\begin{proposition} \label{prop:decompositionOfS}
    Let $S_{I,\eta}$ be the $\mathfrak{gl}_N$-submodule of $S$ generated by $h_{I,\eta}$.  Then $S=\bigoplus_{I,\eta} S_{I,\eta}$ as $\mathfrak{gl}_N$-modules. 
    
    Thus $S$ is isomorphic to a multiple of a highest weight representation: $$S\cong L^{\oplus 2^{N(N+1)/2} }_\nu,\qquad \nu = (N-1/2,N-3/2,\dots,3/2,1/2).$$
\end{proposition}

\begin{proof}
    Fix an ordering on the set of pairs $(p,q)$ so that $(p,q)<(q',p')$ if $p<q$ and $p'<q'$, and so that $(p',p')<(p,q)$ for any $p'$ and any $p<q$.  Then the set of ordered products $\myd_{\alpha_1}\cdots \myd_{\alpha_k}$ where $\alpha_j < \alpha_{j+1}$ for all $j$ as $k$ ranges from $0$ to $N^2$ form a linear basis for $\Cl(W,B)=S$.  It is not hard to see that in this case each $h_{I,\eta}$ is, up to a constant multiple and sign, one of such monomials.  Thus the highest weight vectors $(h_{I,\eta} : I\subset \{1,\dots,N\}, \eta:\{(p,q):p<q\}\to \{0,1\})$ are linearly independent. The dimension of their span is the number of such vectors, which is $2^{N}\cdot 2^{|\{(p,q):1\leq p<q\leq N|\}|} = 2^{N}2^{N(N-1)/2} = 2^{N(N+1)/2}$.

    The natural map \( \bigoplus_{I,\eta} S_{I,\eta}\to S\)  is injective. Indeed, if its kernel were nonzero, semisimplicity would imply that the kernel contains a highest weight vector. Such a vector would give a nontrivial linear relation among the highest weight vectors \(h_{I,\eta}\), which is a contradiction. 

    Finally, 
    \begin{eqnarray*}
        \dim \bigoplus_{I,\eta} S_{I,\eta} &=& 2^{N}\cdot  2^{|\{(p,q):1\leq p<q\leq N\}|}\cdot \dim S_{I,\eta} 
        \\ 
    &=& 2^{N} 2^{N(N-1)/2} 2^{N(N-1)/2} = 2^{N+N(N-1)}=2^{N^2} \\
    &=& \dim S.
    \end{eqnarray*}
     It follows that $S$ must equal the direct sum of $S_{I,\eta}$.  
\end{proof}

\subsection{The  spectrum of $A_N$}

Let \(V=\mathbb{C}^N\), and let \(V^*\) be its dual with the contragredient
representation $\pi_{V^*}$ of \(\mathfrak{gl}_N\). If \(\varepsilon^1,\dots,\varepsilon^N\) is the
dual basis, then
\[
        \pi_{V^*}(E_{ab})\varepsilon^j=-\delta_{aj}\varepsilon^b.
\]
Equivalently, as an ordinary matrix on the dual basis,
\[
        \pi_{V^*}(E_{ab})=-e_{ba}.
\]

Define
\[
        Y=\sum_{a,b=1}^N \pi_{V^*}(E_{ab})\otimes\myd_{ba}
        = -\sum_{a,b=1}^N e_{ba}\otimes\myd_{ba} 
        \in \operatorname{End}(V^*\otimes S).
\]
Thus $$A_N = -\frac{1}{\sqrt{N}} Y.$$ (up to identifying $V$ and $V^*$ using the selected bases).  In particular, $A_N^2 = \frac{1}{N}Y^2$, so that the spectra of $A_N^2$ and $Y^2$ are related by a homothety by ${N^{-1}}$.

\begin{lemma}
    As an element of $\operatorname{End}(V^*\otimes S)$,  $Y^2$  is given by \begin{eqnarray*}
    Y^2 &=&  -2\sum_{a,b=1}^N \pi_{V^*}(E_{ab})\otimes \pi_{S}(E_{ba})
    \\ 
    &=& 2\sum_{a,b=1}^N e_{ba}\otimes F_{ba}.
    \end{eqnarray*}
\end{lemma}
\begin{proof}
    Since $Y=\sum_{a,b=1}^N \pi_{V^*}(E_{ab}) \otimes\myd_{ba} = -\sum_{a,b=1}^N e_{ba}\otimes\myd_{ba} $, we get that
    \begin{eqnarray*} Y^2 &=& \sum_{a,b,c,d} e_{ba} e_{dc} \otimes\myd_{ba} \myd_{dc} 
    = \sum_{a,b,c,d} \delta_{ad} e_{bc}\otimes \myd_{ba}\myd_{dc} \\
    &=& \sum_{b,c}  e_{bc} \otimes \sum_{a} \myd_{ba}\myd_{ac}
    = 2\sum_{b,c}  e_{bc} \otimes F_{bc}\\
    &=& -2\sum_{b,c} \pi_{V^*}(E_{cb}) \otimes F_{bc}
    \end{eqnarray*} which is the claimed expression up to relabeling indices.
\end{proof}

\begin{lemma} \label{lemma:Y2IsCasimir}
    For a representation $\pi_H : \mathfrak{gl}_N \to \operatorname{End}(H)$, denote by $C_H$ the quadratic Casimir element \[ C_H = \sum_{ab} \pi_H(E_{ab}) \pi_H(E_{ba})\in\operatorname{End}(H).\] Then \[ Y^2=C_{V^*}\otimes 1+ 1\otimes C_S -C_{V^*\otimes S}.\]
\end{lemma}
\begin{proof}
    Since $\pi_{V^*}(E_{ab})=-e_{ba}$, $\pi_S(E_{ab})=F_{ab}$, we get: 
    \begin{eqnarray*}
         \pi_{V^*\otimes S}(E_{ab}) &=& \pi_{V^*}(E_{ab}) \otimes 1 + 1\otimes \pi_{S}(E_{ab})  \\
         &=& -e_{ba}\otimes 1 + 1\otimes F_{ab}.
    \end{eqnarray*}
Using this and the previous Lemma gives us: \begin{eqnarray*}
        C_{V^*}\otimes 1 &=& \sum_{ab} e_{ba}e_{ab}\otimes 1 \\
        1\otimes C_{S} &=& \sum_{ab} 1\otimes F_{ab} F_{ba} \\
        C_{V^*\otimes S} &=& \sum_{ab} (-e_{ba}\otimes 1+1\otimes F_{ab})(-e_{ab}\otimes 1+1\otimes F_{ba}) \\
        &=& \sum_{ab} e_{ba}e_{ab}\otimes 1 - e_{ba}\otimes F_{ba} - e_{ab}\otimes F_{ab}
        + 1\otimes F_{ab}F_{ba} \\
        &=& C_{V^*}\otimes 1 -  Y^2 + 1\otimes C_S.
    \end{eqnarray*}
    Solving the last equation for $Y^2$ gives $ Y^2=-(C_{V^*\otimes S}-C_{V^*}\otimes 1-1\otimes C_S)$.
\end{proof}

\begin{lemma} \label{lemma:casimirSpectrum}
    Let $\nu =  (N-1/2,N-3/2,\dots,3/2,1/2)$. Then the spectrum of the Casimir element $C_{V^*}\otimes 1 + 1\otimes C_{L_\nu} - C_{V^*\otimes L_\nu } $ acting on $V^*\otimes L_\nu$ is given by $$
    \sigma(C_{V^*}\otimes 1 + 1\otimes C_{L_\nu} - C_{V^*\otimes L_\nu })= \{1,5,\dots,4N-3\}
    $$
\end{lemma}
\begin{proof}
If $L_\mu$ is the irreducible representation of $\mathfrak{gl}_N$ with highest weight $\mu=(\mu_1,\dots,\mu_N)$, then the Casimir element $C_{L_\mu}$ acts as the following scalar multiple of identity:
$$
C_{L_\mu} = \sum_{t=1}^N \mu_t (\mu_t + N+1-2t) \operatorname{id}_{L_\mu}.
$$ (the computation is similar to the one in  \cite[\S3.3.2, Lemma 3.3.8]{GoodmanWallach}, although \cite{GoodmanWallach} uses a slightly different Casimir element adapted to $\mathfrak{sl}_N$ and its Killing form: $C'=(2N)^{-1}(\sum_{ab} E_{ab}E_{ba} - N^{-1}(\sum_{a}E_{aa})^2)$ as opposed to our $C=\sum_{ab} E_{ab}E_{ba}$. In the representation $L_\mu$, $\sum_{a} E_{aa}$ acts as the $(\sum \mu_i )\operatorname{id}_{L_\mu}$  and so $C'$ differs from $C_{L_\mu}$  by an overall factor and a shift by a constant multiple of identity. Since any Casimir acts as some multiple of identity in any irreducible representation, the exact multiple can be obtained by explicitly computing the action on any vector, e.g., the highest weight vector.)

The representation $V^*$ is irreducible, and $$C_{V^*}=N \operatorname{id}_{V^*}$$ (by direct computation, or using that its weight is $(0,\dots,0,-1)$). 

With $\nu =  (N-1/2,N-3/2,\dots,3/2,1/2)$, $C_{L_\nu}$ acts as the scalar $$
C_{L_\nu} = \sum_{t=1}^N (N+\frac{1}{2}-t)(2N+\frac{3}{2}-3t ).$$

By Pieri's rule (see e.g. \cite[\S9.2, Corollary 9.2.4]{GoodmanWallach}), $$V^*\otimes L_\nu = \bigoplus_{t=1}^{N} L_{\nu-\epsilon_t}$$ where $\epsilon_t = (0,\dots,1,\dots,0)$ (non-zero entry in the $t$-th place; note that $\nu-\epsilon_t$ still has monotonically non-increasing entries and thus remains a dominant weight). Hence the Casimir $C_{V^*\otimes L_\nu}$ acts on the $r$-th irreducible component as the scalar \[ c_r = 
\sum_{t=1}^N (N+\frac{1}{2} -t -\delta_{rt})(2N+\frac{3}{2}-3t -\delta_{rt}).\]

Thus, restricting to the $r$-th irreducible component of $V^*\otimes L_\nu$, $C_{V^*}\otimes 1 + 1\otimes C_{L_\nu} - C_{V^*\otimes L_\nu } $ acts as the scalar $y_r$, given by 
\begin{eqnarray*}
    y_r &=& N + \sum_{t=1}^N (N+\frac{1}{2}-t)(2N+\frac{3}{2}-3t) 
    \\ && - \sum_{t=1}^N (N+\frac{1}{2}-t -\delta_{rt})(2N+\frac{3}{2}-3t-\delta_{rt}) \\
    &=& N + (N+\frac{1}{2}-r)(2N+\frac{3}{2}-3r) 
    \\ && - (N+\frac{1}{2}-r-1)(2N+ \frac{3}{2}-3r-1) \\
    &=& N + 2N+\frac{3}{2}-3r + N+ \frac{1}{2}-r - 1 \\
    &=& 4N-4r+1.
\end{eqnarray*}
\end{proof}

\begin{theorem}\label{thrm:spectrumOfY}
    The spectrum of $Y^2$ is the set $\{1,5,\dots,4N-3\}$. 
\end{theorem}
\begin{proof}
By Proposition~\ref{prop:decompositionOfS} the representation $S$ is a multiple of the irreducible representation $L_\nu$ with $\nu$ as in Lemma~\ref{lemma:casimirSpectrum}. Thus  $V^*\otimes S$ is a multiple of $V^*\otimes L_\nu$.   

It follows that the spectrum of  $C_{V^*} \otimes 1 + 1\otimes C_{L_\nu} - C_{V^*\otimes L_\nu}$ (given by Lemma~\ref{lemma:casimirSpectrum}) is the same as the spectrum of $C_{V^*}\otimes 1 + 1\otimes C_{S} - C_{V^*\otimes S}$, which is equal to $Y^2$ by Lemma~\ref{lemma:Y2IsCasimir}.
\end{proof}

\begin{theorem}\label{thrm:spectrumOfA}
    The spectrum of $A_N$ is the set  $$\{-\sqrt{4-3/N},\dots,-\sqrt{5/N},-\sqrt{1/N},\sqrt{1/N},\sqrt{5/N},\dots,\sqrt{4-3/N}\}.$$
    In particular, $\Vert A_N\Vert = \sqrt{4-3/N} <2$ and $\Vert A_N\Vert\to 2$ as $N\to\infty$.
\end{theorem}
\begin{proof}
    The map $\myd_{pq}\mapsto -\myd_{pq}$ is an automorphism of the Clifford algebra preserving the trace; thus the spectra of $A_N$ and $-A_N$ are the same, and so the spectrum is symmetric under negation.  Since $A_N=-N^{-1/2}Y$, $A_N^2 = N^{-1}Y^2$, and thus $\sigma(A_N)=\{\pm\sqrt{\lambda/N}:\lambda\in\sigma(Y^2)\}$. We now apply Theorem~\ref{thrm:spectrumOfY}.
\end{proof}

\begin{corollary}
Let \(A_N^{(1)},\ldots,A_N^{(M)}\) be as above, and let \(A^{(1)},\ldots,A^{(M)}\) be a free semicircular system
of variance \(1\). Then for every
\(\lambda_1,\ldots,\lambda_M\in \mathbb R\),
\[
\lim_{N\to\infty}
\left\|
\sum_{r=1}^M \lambda_r A_N^{(r)}
\right\|
=
\left\|
\sum_{r=1}^M \lambda_r A^{(r)}
\right\|
=2\left(\sum_{r=1}^M |\lambda_r|^2\right)^{1/2}.
\]
\end{corollary}

\begin{proof}
Let $\Lambda = (\sum_{r=1}^M |\lambda_r|^2)^{1/2}$.  The map that sends the $i,j$-th entry $\sum \lambda_r A_N^{(r)}$ to that of $\Lambda\cdot  A_N$ is an injective $*$-homomorphism of the  $C^*$-algebras generated by the entries.  Thus $\Vert \sum \lambda_r A_N^{(r)}\Vert = \Lambda \Vert A_N\Vert$.  We now apply Theorem~\ref{thrm:spectrumOfA}.
\end{proof}

\subsection{The spectral measure of $A_N$}\label{subsec:SpectralMeasureAN}

While we do not need this for the rest of the paper, we give an exact expression for the spectral measure of $A_N^2$ with respect to the state $\phi_N$.  Indeed the spectral measure is the weighted sum of delta functions $$\mu_{A_N^2}=\sum_{r=0}^{N-1} w_r \delta_{(4r+1)/N}.$$ The weight $w_r$ is the multiplicity of the $r$-th eigenvalue normalized by the dimension of the entire space. 
Since the $r$-th eigenvalue arises from $L_{\nu-\epsilon_{N-r}}$ viewed as an irreducible component of $V^*\otimes L_\nu$, we see that  $$w_r =  \frac{2^{N(N+1)/2}\dim L_{\nu-\epsilon_{N-r}}}{2^{N(N+1)/2}\dim V^*\otimes L_\nu}
= \frac{1}{N} \frac{\dim L_{\nu-\epsilon_{N-r}}}{\dim L_\nu}.$$  Using Weyl's dimension formula we end up with the expression
$$w_r = \frac{1}{4^{N-1}} G(r,N-1-r)$$ with
$$G(r,s)=\frac{2s+1}{r+s+1}\binom{2r}{r}\binom{2s}{s}.$$
In particular, this leads to the moment formula
\[
\phi_{N}(A^{2p}_N) =  \frac{1}{4^{N-1}} \frac{1}{N^p} \sum_{r=0}^{N-1} G(r,N-1-r)(4r+1)^p .
\]
Using a standard genus expansion argument for random matrices one can interpret the moments $\phi_N(A_N^{2p})$ as signed sums over pair-partitions (see \cite{Shlyakhtenko:Matrices}). This leads to the following combinatorial identity, valid for $p,N\geq 1$:
\begin{equation}\label{eq:MomentsAreSuperBallots}
\sum_{\pi \in \mathcal{P}_2(2p)} (-1)^{\operatorname{cr}(\pi)} N^{\#(\gamma\pi)} = \frac{N}{4^{N-1}} \sum_{r=0}^{N-1} G(r,N-1-r) (4r+1)^p.
\end{equation}
Here each $\pi$ is a pair-partition, $\operatorname{cr}(\pi)$ is the number of crossings of $\pi$, and $\#(\gamma\pi)$ is the number of cycles of $\gamma \pi$.  We note that the numbers $G(r,s)$ have appeared in the work of Gessel \cite{gessel} on super ballot numbers.

\section{Failure of strong convergence}
In this section we show that already for linear combinations of words of length $2$ in the matrices $A_N^{(r)}, B_N^{(r)}$ convergence of operator norms may fail.

We start with a few preliminary computations. Fix $M,N\geq 1$ and consider the operators $Q_N^{(k)}$, $k=1,\dots,M$, acting on the tensor product $\mathbb{C}^N\otimes \mathscr{F}$ of $\mathbb{C}^N$ and the antisymmetric Fock space $\mathscr{F}$. 

\subsection{The vector $\Omega_{J,M}$}
Let \[
d^{(k)}_{pq}:=\frac12\bigl(b(h^{(1,k)}_{pq})+i b(h^{(2,k)}_{pq})\bigr)
\] so that 
\[ Q_N^{(k)} = \frac{2}{\sqrt{N}}\sum_{p,q} e_{pq} \otimes  d^{(k)}_{pq}.\]

Let 
\[
    \Omega_{J,M}
    =
    2^{-MN^2/2}
    \prod_{k=1}^M\prod_{p,q=1}^N
    \left(
      1-i\,a(h^{(1,k)}_{pq})^*a(h^{(2,k)}_{pq})^*
    \right)\Omega,
\]
The order of the product defining $\Omega_{J,M}$ does not matter, since the operators $a(h^{(1,k)}_{pq})^*$ anticommute, and thus the operators $a(h^{(1,k)}_{pq})^*a(h^{(2,k)}_{pq})^*$ commute, for different triples $(p,q,k)$.

The next Lemma shows that the vector $\Omega_{J,M}$ behaves like a vacuum vector for the operators $d_{pq}^{(k)}$, which themselves turn out to form a CAR family.

\begin{lemma} \label{lemma:propertiesOfOmega}
    With the above notation, the following properties hold: 
    \begin{enumerate}[label={\rm(\roman*)}]
    \item $\Omega_{J,M}$ is a unit vector 
    \item For every \(k,p,q\), \(
    d^{(k)}_{pq}\Omega_{J,M}=0. \) 
    \item For every \(k,p,q,r\), \(
    d^{(k)}_{rq}(d^{(k)}_{pq})^*\Omega_{J,M}
    =
    \delta_{rp}\Omega_{J,M}.
    \)
    \end{enumerate}
\end{lemma}
\begin{proof}
Let \[\Omega_{J, p,q,k}= \frac{1}{\sqrt{2}}\left(
      1-i\,a(h^{(1,k)}_{pq})^*a(h^{(2,k)}_{pq})^*
    \right)\Omega = \frac{1}{\sqrt{2}}\left(\Omega - i h^{(1,k)}_{pq}\wedge h^{(2,k)}_{pq}\right).\] It is clear that 
\(\Omega_{J, p,q,k}\) is a unit vector, and is an even-parity vector in the 
Fock space. 
Under the canonical graded tensor product decomposition of the Fock space
over the orthogonal direct sum of the two-dimensional spaces
\[
\operatorname{span}\{h_{pq}^{(1,k)},h_{pq}^{(2,k)}\},
\]
the vector \(\Omega_{J,M}\) is the tensor product of the unit even-parity vectors
\(\Omega_{J,p,q,k}\).  Since the tensor products of orthogonal unit vectors are unit vectors, we obtain (i).

For part (ii), note that $d^{(k)}_{pq}$ commutes with $1-i\,a(h^{(1,k')}_{p'q'})^*a(h^{(2,k')}_{p'q'})^*$ unless $p=p'$, $q=q'$, $k=k'$.  It follows that 
\[d^{(k)}_{pq}\Omega_{J,M} = \left(d^{(k)}_{pq} \Omega_{J,p,q,k}\right) 
\wedge \left(\bigwedge_{(p',q',k')\neq (p,q,k)} 
\Omega_{J,p',q',k'}\right), \] where the order of the antisymmetric tensor products is again irrelevant since all but one of the terms have even parity.

But
\begin{eqnarray*}
    2\sqrt{2} d^{(k)}_{pq} \Omega_{J,p,q,k} &=& 
    \bigl(b(h^{(1,k)}_{pq})+i b(h^{(2,k)}_{pq})\bigr) \bigl(\Omega - i h^{(1,k)}_{pq}\wedge h^{(2,k)}_{pq}\bigr) \\ &= & 
    h^{(1,k)}_{pq} + i h^{(2,k)}_{pq} \\
    && - h_{pq}^{(1,k)}\wedge h^{(1,k)}_{pq}\wedge h^{(2,k)}_{pq} - i h_{pq}^{(2,k)} \\
    &&  - i h_{pq}^{(2,k)} \wedge h^{(1,k)}_{pq}\wedge h^{(2,k)}_{pq}
     - h_{pq}^{(1,k)}  \\ &=& 0.
\end{eqnarray*} Therefore, (ii) holds. 

For part (iii), note that 
\[
d_{pq}^{(k)}
=
\frac12\Big(
a(h^{(1,k)}_{pq})
+a(h^{(1,k)}_{pq})^*
+i\,a(h^{(2,k)}_{pq})
+i\,a(h^{(2,k)}_{pq})^*
\Big).
\] 

We claim that 
\begin{equation} \label{eqn:CARd}
\{(d^{(k)}_{rq}), (d^{(k)}_{pq})^*\} = \delta_{rp} 1.
\end{equation}

If $r\neq p$, the operators $d_{rq}^{(k)}$ and $(d_{pq}^{(k)})^*$ anti-commute, since the corresponding expressions in the creation/annihilation operators $a^*$ and $a$ anti-commute, and thus \eqref{eqn:CARd} holds true. 

If $r=p$, using CAR for the operators $a$ gives:
\[\{ a(h_{pq}^{(\alpha,k)}), a(h_{pq}^{(\alpha',k)})\} = 0,\qquad \{ a(h_{pq}^{(\alpha,k)}), a(h_{pq}^{(\alpha',k)})^*\} = \delta_{\alpha\alpha'}1,\]
Thus
\begin{eqnarray*}
    \{(d^{(k)}_{pq}), (d^{(k)}_{pq})^*\} & = & 
    \frac{1}{4}\{a(h^{(1,k)}_{pq}) +a(h^{(1,k)}_{pq})^*
+i\,a(h^{(2,k)}_{pq})
+i\,a(h^{(2,k)}_{pq})^*    
 , \\ && \qquad 
a(h^{(1,k)}_{pq})^*
+a(h^{(1,k)}_{pq})
-i\,a(h^{(2,k)}_{pq})^*
-i\,a(h^{(2,k)}_{pq})
\} \\ &= & \frac{1}{4}(1 + 1 - (i \cdot i) - (i\cdot i)) = 1.
\end{eqnarray*}
This proves \eqref{eqn:CARd} in the remaining case.  

Rewriting \eqref{eqn:CARd} gives $d^{(k)}_{rq}(d^{(k)}_{pq})^* = \delta_{rp}1 - (d^{(k)}_{pq})^* d^{(k)}_{rq}
$, and so 
 \[d^{(k)}_{rq}(d^{(k)}_{pq})^*\Omega_{J,M} = 
\bigl(\delta_{rp}1 - (d^{(k)}_{pq})^* d^{(k)}_{rq}\bigr)\Omega_{J,M} = \delta_{rp} \Omega_{J,M}\] as claimed. 
\end{proof}

Let now 
\[
\eta_0 = \frac{1}{\sqrt{N}} \sum_{p=1}^N e_p \otimes \Omega_{J,M}
\] where $e_p$ denotes the column vector all of whose entries are zero, except that the $p$-th entry is $1$. 

\begin{lemma} \label{lemma:QOnEta} For each $k$,
    \[(Q^{(k)}_N)^* Q^{(k)}_N\eta_0 = 0,  \qquad Q^{(k)}_N (Q^{(k)}_N)^* \eta_0 = 4 \eta_0.\]
\end{lemma}

\begin{proof}
    Since $Q^{(k)}_N = 2 N^{-1/2} \sum_{p,q} e_{pq}\otimes d_{pq}^{(k)}$, we can apply Lemma~\ref{lemma:propertiesOfOmega} to deduce
    \[Q_N^{(k)} \eta_0 = \frac{2}{N}\sum_{pq} e_{p}\otimes  d_{pq}^{(k)} \Omega_{J,M} = 0  \] Thus also \( (Q^{(k)}_N)^* Q^{(k)}_N\eta_0 = 0\).

    On the other hand,
    \begin{eqnarray*}
        Q^{(k)}_N (Q^{(k)}_N)^* \eta_0 &=& 
        \frac{4}{N^{3/2}}\sum_{p,q,r} e_{r} \otimes d_{rq}^{(k)} (d_{pq}^{(k)})^* \Omega_{J,M} \\
        &=& \frac{4}{N^{3/2}}\sum_{p,q,r} e_{r} \otimes \delta_{rp}  \Omega_{J,M}  \\
        &=& \frac{4}{N^{3/2}}\sum_{q,r} e_{r} \otimes \Omega_{J,M}  
        =4\eta_0,
    \end{eqnarray*} where we used Lemma~\ref{lemma:propertiesOfOmega} to pass from the first to the second line.
\end{proof}

\subsection{Norms of sums of squares.}
Let $$R_N = \begin{bmatrix} Q_N^{(1)} & Q_N^{(2)} & \cdots & Q_N^{(M)} \\
0 & 0 & \cdots & 0 \\
\vdots & \vdots  & & \vdots \\
0 & 0 & \cdots & 0 \end{bmatrix} \in M_{M\times M}(\mathbb{C})\otimes \Cl(W,B). $$  Then $$
R_N R_N^*  = 1\otimes \sum_{j=1}^M Q_N^{(j)} (Q_N^{(j)})^*,$$ and so using Lemma~\ref{lemma:QOnEta}, $$\Vert R_N R_N^*\Vert \geq \frac{1}{\Vert \eta_0 \Vert} \sum_{j=1}^M \Vert Q_N^{(j)} (Q_N^{(j)})^* \eta_0 \Vert = \frac{4M\Vert\eta_0\Vert}{\Vert\eta_0\Vert}=4M.
$$
It follows that $$\liminf_{N\to\infty} \Vert R_N \Vert \geq 2\sqrt{M}. $$

On the other hand, let $A^{(1)},B^{(1)},\dots,A^{(M)},B^{(mM}$ be free semicircular elements, and let $Q^{(j)}=A^{(j)} + i B^{(j)}$.  be  circular elements (these are $\sqrt{2}$ times the usual normalization of a circular element).  Let $$
R = \begin{bmatrix} Q^{(1)} & Q^{(2)} & \cdots & Q^{(M)} \\
0 & 0 & \cdots & 0 \\
\vdots & \vdots  & & \vdots \\
0 & 0 & \cdots & 0 \end{bmatrix} \in M_{M\times M}(\mathbb{C})\otimes C^*(Q^{(1)},\dots,Q^{(M)}).
$$
Then $$RR^* = \sum_{j=1}^M Q^{(j)} (Q^{(j)})^*$$ is an $M$-fold sum of $\lambda=1$ free Poisson elements $Q^{(j)}(Q^{(j)})^*$ each having a law with support $[0,8]$ and density $$
\frac{1}{4\pi x}\sqrt{(8-x)x}\,
\mathbf 1_{[0,8]}(x).
$$  
Thus $RR^*$ is a $\lambda=M$  free Poisson element having its law supported on the interval 
$[2(\sqrt{N}-1)^2,2(\sqrt{M}+1)^2]$. 
In particular $$\Vert R\Vert = \sqrt{2}(\sqrt{M}+1).$$

Noting that if $M\geq 6$, then $2\sqrt{M} > \sqrt{2}(\sqrt{M}+1)$, we can record the above computations as the following: 
\begin{lemma}
    \label{lemma:WrongOperatorSpaceStructure}
    With the above notation, 
    \begin{eqnarray} && \Vert R \Vert = \Big\Vert \sum_{j=1}^M Q^{(j)} (Q^{(j)})^* \Big\Vert^{1/2} = \sqrt{2}(\sqrt{M}+1),\\ && \liminf_{N\to\infty}  \Vert R_N \Vert = 
    \liminf_{N\to\infty} \Big\Vert \sum_{j=1}^M Q_N^{(j)} (Q_N^{(j)})^*\Big\Vert^{1/2} \geq 2\sqrt{M}.\end{eqnarray} xw
    In particular, if $M\geq 6$, then $2\sqrt{M}>\sqrt{2}(\sqrt{M}+1)$ and so \begin{eqnarray}
        &&\liminf_{N\to\infty} \Vert R_N\Vert \neq \Vert R \Vert,\label{eq:noOpConv}\\
        &&\liminf_{N\to\infty}\Big\Vert \sum_{j=1}^M Q_N^{(j)} (Q_N^{(j)})^*\Big\Vert \neq \Big\Vert \sum_{j=1}^M Q^{(j)} (Q^{(j)})^* \Big\Vert.\label{eq:noStrongConv}
    \end{eqnarray} 
\end{lemma}
We record an immediate corollary: 

\begin{theorem}
   \label{thrm:noStrong}
    Let $M\geq 6$.  Then: 
    \begin{itemize}
    \item[(i)] the asymptotic operator space structure of  $$\operatorname{span}_\mathbb{C}(A_N^{(1)},B_N^{(1)},\dots,A_N^{(M)},B_N^{(M)})$$  is different from that of the complex span of $2M$ free sesmicircular variables.  
    \item[(ii)] The family $(A_N^{(1)},B_N^{(1)},\dots,A_N^{(M)},B_N^{(M)})$ converges to the free semicircular family $(A^{(1)},B^{(1)},\dots,A^{(M)},B^{(M)})$ in law but not strongly. In fact, strong convergence already fails for the quadratic polynomial $$P(a_1,b_1,\dots,a_m,b_m)=\sum_{j=1}^M  (a_j + i b_j)(a_j + b_j)^*$$
    \end{itemize}
\end{theorem}
\begin{proof}
    The first part follows by  noting that $$R_N =  \sum_{j=1}^M e_{1j} \otimes A_N^{(j)} + \sum_{j=1}^M (i e_{1j})\otimes B_N^{(j)}$$ and by applying \eqref{eq:noOpConv}.

    The second part follows by evaluating $P$ and applying \eqref{eq:noStrongConv}.
\end{proof}

\subsection{Norms of commutators}

\begin{lemma} \label{lemma:normOfCommutator}
    Let \(C_N^{(k)} = i [ A_N^{(k)} , B_N^{(k)}]\). 
                  Then \(
                  \Vert \sum_{k=1}^M C_N^{(k)}\Vert \geq 2M
                  \).
\end{lemma}
\begin{proof}
Note that since $ A_N^{(k)} = \frac{1}{2}(Q_{N}^{(k)} + (Q_N^{(k)})^*)$
and $B_N^{(k)} = \frac{1}{2i}(Q_{N}^{(k)} - (Q_N^{(k)})^*)$, 
\[
C_N^{(k)} = i [ A_N^{(k)} , B_N^{(k)} ] 
= \frac12\left((Q_N^{(k)})^*Q_N^{(k)}
                  -Q_N^{(k)}(Q_N^{(k)})^*\right).
\]
        
Therefore by Lemma~\ref{lemma:QOnEta}, for any $k=1,\dots,M$,
\[
C^{(k)}_N \eta_0 = \frac12\left((Q_N^{(k)})^*Q_N^{(k)}
                  -Q_N^{(k)}(Q_N^{(k)})^*\right) \eta_0 = -2\eta_0. 
\]
It follows that
\(
 \sum_{k=1}^M C^{(k)}_N \eta_0 = -2M \eta_0
\) so that 
\[
\left\Vert \sum_{k=1}^M C_N^{(k)} \right\Vert \geq 2M,
\] as claimed. \end{proof}

\begin{lemma}\label{lemma:freeNorms}
    Let $A^{(k)}, B^{(k)}$, $k=1,\dots,M$ be free semicircular elements of variance $1$, and let $C^{(k)} = i [A^{(k)},B^{(k)}]$.  Then $\Vert \sum_{k=1}^M C^{(k)} \Vert \leq 8  + 2\sqrt{2M}$.
\end{lemma}

\begin{proof}
    Since $A^{(k)}, B^{(k)}$ are free semicircular elements, $\Vert A^{(k)}\Vert = \Vert B^{(k)}\Vert = 2$, and thus $\Vert C^{(k)}\Vert \leq 2 \Vert A^{(k)}\Vert \Vert B^{(k)}\Vert = 8$. A direct computation   shows that $\Vert C^{(k)}\Vert_{L^2}^2 = \tau(-A^{(k)}B^{(k)}A^{(k)}B^{(k)}-B^{(k)}A^{(k)}B^{(k)}A^{(k)}+A^{(k)}B^{(k)}B^{(k)}A^{(k)}+B^{(k)}A^{(k)}A^{(k)}B^{(k)})= 2$.

    Since $C^{(1)},\dots,C^{(M)}$ are freely independent and centered, it follows from Voiculescu's inequality for norms of freely independent centered variables (see \cite[Lemma 3.2]{Voiculescu:Norm}) that 
    \[ \left\Vert \sum_{k=1}^M C^{(k)} \right\Vert   \leq \max_{k} \Vert C^{(k)} \Vert + 2\left(\sum_{k=1}^M \Vert C^{(k)}\Vert_{L^2}^2  \right)^{1/2}
     \leq 8  + 2\sqrt{2M},\] 
    as claimed. 
\end{proof}

In this way we again obtain a version of~\ref{thrm:noStrong}:

\begin{theorem}\label{thrm:noStrongWeaker}
Let $M\geq 9$. Then the family $(A_N^{(1)},B_N^{(1)},\dots,A_N^{(M)},B_N^{(M)})$ converges to the free semicircular family $(A^{(1)},B^{(1)},\dots,A^{(M)},B^{(M)})$ in law but not strongly.  In fact,
\[
\liminf_{N\to \infty } \left\Vert \sum_{k=1}^M i[A^{(k)}_N, B^{(k)}_N] \right\Vert \geq 2M > 8 + 2\sqrt{2M} \geq \left\Vert \sum_{k=1}^M i[A^{(k)}, B^{(k)}]\right\Vert.
\]
\end{theorem}

\begin{proof}
The stated inequalities are direct consequences of Lemma~\ref{lemma:normOfCommutator}, the fact that $2M > 8 + 2\sqrt{2M}$ for $M\geq 9$, and Lemma~\ref{lemma:freeNorms}.  Convergence in law holds because of \cite{Shlyakhtenko:Matrices}.  Strong convergence would imply that for any non-commutative polynomial $P$, 
\[\lim_{N\to\infty} \left\Vert P(A_N^{(1)},B_N^{(1)},\dots,A_N^{(M)},B_N^{(M)})\right\Vert = 
\left\Vert P(A^{(1)},B^{(1)},\dots,A^{(M)},B^{(M)})\right\Vert, \] which fails with 
$P(a_1,b_1,\dots,a_M,b_M)=\sum_{k=1}^M i[a_k,b_k]$.
\end{proof}

\subsection{Further remarks}\label{sec:OtherOmeags}
It may be tempting to ask whether it is possible to ``remove'' the extraneous eigenvectors of the commutators $C^{(k)}_N$ and deduce convergence on a suitable subspace (by analogy with strong convergence of permutation matrices, where one needs to restrict to the orthocomplement of the all-ones vector). To this end, we note that there are further counterexamples that can be constructed, in the following way.

Fix $\epsilon : \{1,\dots,M\} \to \{\pm 1\}$, and let\[
    \Omega_{J,M,\epsilon}
    =
    2^{-MN^2/2}
    \prod_{k=1}^M\prod_{p,q=1}^N
    \left(
      1-i\epsilon(k)\,a(h^{(1,k)}_{pq})^*a(h^{(2,k)}_{pq})^*
    \right)\Omega,
\]
(in this way our previous construction corresponds to $\epsilon$ being identically equal to $1$).  Furthermore, fix any unitary $u$ in the commutant of the $C^*$-algebra $C^*(d^{(k)}_{pq}):1\leq k\leq M, 1\leq p,q\leq N)$ viewed as an algebra of operators on $\mathscr{F}$, and let $w\in \mathbb{C}^N$ be an arbitrary unit vector.  Set 
$$
\eta = w\otimes (u\Omega_{J,M,\epsilon}).$$ Then a similar computation shows that 
$C_N^{(k)} \eta = -2 \epsilon(k)\eta$ so that $$\left\Vert\sum_{k=1}^M \epsilon(k) C_N^{(k)}\right\Vert \geq 2M,$$  while the norm of the operator $\sum \epsilon(k)C^{(k)}$ does not depend on $\epsilon$. In other words, strong convergence cannot happen after ``removal'' of just a few ``bad vectors''.

\section{The $q$-Gaussian case, $-1<q<1$: Strong Convergence.}
Bo\.zejko and Speicher constructed a one-parameter family of $q$-deformed creation/annihilation operators \cite{Bozejko-Speicher:qBrownian}. 
In this section, we consider the $q$-deformed analogs of the CAR matrices we considered above.  

We briefly recall the construction to set the notation.  Let $H_\mathbb{C}$ be a complex Hilbert space, and consider the algebraic tensor product 
\[
\mathscr{F}_q^{\operatorname{alg}} = \mathbb{C}\Omega \oplus \bigoplus_{k\geq 1} 
(H_\mathbb{C})^{\atimes k} 
\]
where we use $\atimes$ to denote the algebraic tensor product.  For a fixed number $-1\leq q\leq 1$, define a scalar product on $\mathscr{F}_q^{\operatorname{alg}}$ by \[
\langle \xi_1\atimes \cdots \atimes \xi_k ,\zeta_1\atimes \cdots \atimes \zeta_m\rangle 
= \delta_{km} \sum_{\pi \in \mathbb{S}_k} (q)^{\operatorname{inv}(\pi)}  \prod_{j=1}^k \langle \xi_j, \zeta_{\pi(j)}\rangle,
\]
where $\mathbb{S}_k$ denotes the group of permutations of the set $\{1,\dots,k\}$, and \[
\operatorname{inv}(\pi)=\bigl|\{(i,j): 1\leq i< j\leq k \textrm{ but } \pi(i)>\pi(j)\}\bigr|\] denotes the number of inversions of a permutation $\pi$. 

By \cite[Proposition 1]{Bozejko-Speicher:qBrownian}, this sesquilinear scalar product is positive semi-definite for all values of $q\in [-1,1]$, and is positive definite the moment $|q|<1$.  (In the extreme cases $q=\pm 1$ one can quotient $\mathscr{F}_q^{\operatorname{alg}}$ by the kernel of the scalar product; in that case one obtains precisely the symmetric ($q=1$) or anti-symmetric ($q=-1$) Fock spaces).

We will denote by $\mathscr{F}_q$ the completion of $\mathscr{F}_q^{\operatorname{alg}}$ (modulo the kernel of the scalar product in cases $q=\pm 1$) with respect to the norm arising from this scalar product.

For $h\in H_\mathbb{C}$ denote by $\ell(h)$ the $q$-creation operator given on $\mathscr{F}_q^{\operatorname{alg}}$ by linearity and 
\[
\ell(h) \xi_1\atimes \cdots \atimes \xi_k = h\atimes \xi_1\atimes \cdots \atimes \xi_k.
\]
The operator $\ell(h)$ turns out to be bounded for all $-1\leq q<1$ and thus extends to all of $\mathscr{F}_q$; we abuse notation and denote it by the same symbol. Its adjoint is then given by the formula
\[
\ell(h)^* \xi_1\atimes\cdots\atimes \xi_k = \sum_{j=1}^{k} q^{j-1} \langle h,\xi_j\rangle \xi_1\atimes\cdots 
\atimes \xi_{j-1} \atimes \widehat{\xi_j} \atimes \xi_{j+1}\atimes\cdots \atimes \xi_k, 
\]
where $\widehat{\cdot}$ denotes omission of a term.  These operators satisfy the following $q$-commutation relations: \begin{equation} \label{eq:qCommutation}
\ell(h)^*\ell(g) = \langle h,g\rangle 1  + q \ell(g) \ell(h)^*,\qquad h,g\in H_\mathbb{C}.
\end{equation}

\subsection{The matrices $A_N^{(k,q)}$, $Q_N^{(k,q)}$ and $L_N^{(k)}$.\label{subseq:qMatrices}}
Let $H$ be, as before, a real Hilbert space with basis $h_{rs}^{(\alpha,k)}$, $\alpha=1,2$, $k=1,\dots,M$ and $1\leq r,s \leq N$. Let
\[
Q_N^{(k,q)} = \frac{1}{\sqrt{N}} \sum_{r,s=1}^N e_{rs} \otimes 
\bigl( b_q (h_{rs}^{(1,k)}) + i b_q(h_{rs}^{(2,k)}))
\]
where $e_{rs}$ are matrix units for the algebra of $N\times N$ matrices, and 
$$
b_q (h) = \ell(h)+\ell(h)^*.$$

Note that \begin{eqnarray*}
    b_q (h_{rs}^{(1,k)}) + i b_q(h_{rs}^{(2,k)}) &=& \ell(h_{rs}^{(1,k)} + i h_{rs}^{(2,k)})  
+ \ell(h_{rs}^{(1,k)} -i  h_{rs}^{(2,k)})^* \\
&=& \sqrt{2}(\ell(g_{rs}^{(1,k)}) + \ell(g_{rs}^{(2,k)})^*)
\end{eqnarray*}
where $ g_{rs}^{(1,k)} = 2^{-1/2}(h_{rs}^{(1,k)} + i  h_{rs}^{(2,k)}) $
and $g_{rs}^{(2,k)} = 2^{-1/2} (h_{rs}^{(1,k)} -i  h_{rs}^{(2,k)})$ are two orthonormal families of vectors. 

Thus if we let \[
L_N^{(\alpha,k)} = \frac{1}{\sqrt{N}} \sum_{r,s=1}^N e_{rs}\otimes \ell ( g_{rs}^{(\alpha,k)}),\qquad \alpha=1,2,\quad k=1,\dots,M,
\]
then \[
Q_N^{(k,q)} = \sqrt{2}(L_N^{(1,k)} + (L_N^{(2,k)})^*),\qquad k=1,\dots,M.
\]
(We omit the superscript $q$ from the matrix $L$ for brevity).  

If we now let 
$$A_N^{(k,q)} = \frac{Q_N^{(k,q)} + (Q_N^{(k,q)})^*}{2}, \qquad 
B_N^{(k,q)} = \frac{Q_N^{(k,q)} - (Q_N^{(k,q)})^*}{2i} $$ then we have, recalling the notation of \S\ref{sect:CARMatrices},
\[
Q_N^{(k)} = Q_N^{(k,-1)},\qquad A_N^{(k)} = A_N^{(k,-1)},\qquad B_N^{(k)} = B_N^{(k,-1)}.
\]

The aim of this section is to prove that, for fixed $q$ in the interval $(-1,1)$,
the families $(A_N^{(k,q)}, B_N^{(k,q)}:k=1,2,\dots,M)$ converge strongly to a free semicircular family $(A^{(k)}, B^{(k)}:1\leq k\leq M)$ as $N\to\infty$ (and thus the family  $(Q_N^{(k,q)}: k=1,2,\dots,M)$ converges to the free circular family $(A^{(k)}+i B^{(k)}: k=1,\dots,M)$).

\subsection{Technical estimates on the operator $\mathscr{L}$.}

We begin with a technical lemma.  We fix $-1<q<1$. 
  Let $e_1,\dots,e_s$ be an orthonormal basis for $\mathbb{C}^s$, and let $g_1,\dots,g_s$ be an orthonormal family of vectors in $H_\mathbb{C}$.  For $R=0,1,2,\dots$, denote by 
    $\mathscr{F}_q^{\leq R}$ the subspace of $\mathscr{F}_q^{\operatorname{alg}}$ consisting of tensors of rank $R$ or less (by convention $\mathscr{F}_q^{\leq 0} = \mathbb{C}\Omega$), and consider the map $$\mathscr{L}_R : \mathbb{C}^s \otimes \mathscr{F}_q^{\leq R}\to\mathscr{F}_q^{\leq R+1} $$ given by
    $$\mathscr{L}_R \left(\sum_{j=1}^s e_j \otimes \xi_j \right )= \sum_{j=1}^s\ell(g_j) \xi_j.$$

\begin{lemma} \label{lem:operatorL}
   Let $c_R = (1 + |q| + \cdots + |q|^{R})^{1/2}$.   Then, with the above notation, \begin{equation}\label{eq:scriptEllStar} 
    \mathscr{L}_R^* (\xi) = \sum_u e_u \otimes \ell(g_u)^* \xi,\qquad \forall \xi\in \mathscr{F}_q^{\leq R+1}
    \end{equation} and  $$\Vert \mathscr{L}_R \Vert \leq c_R \leq \frac{1}{\sqrt{1-|q|}}.$$  In particular, there exists a bounded operator $\mathscr{L}:\mathbb{C}^s\otimes \mathscr{F}_q\to\mathscr{F}_q$ whose restriction to  $\mathbb{C}^s\otimes \mathscr{F}_q^{\leq R}$ is $\mathscr{L}_R$ for each $R$. Moreover,  
    \begin{equation}\label{eq:scriptEllStarClosure} 
    \mathscr{L}^* (\xi) = \sum_u e_u \otimes \ell(g_u)^* \xi,\qquad \forall \xi\in \mathscr{F}_q
    \end{equation} and $\Vert \mathscr{L}\Vert  \leq (1-|q|)^{-1/2}$. 
    
\end{lemma}
\begin{proof}
     If $\sum_t e_t \otimes \zeta_t \in \mathbb{C}^s \otimes \mathscr{F}_q^{\leq R}$, then for any $\xi\in \mathscr{F}_q^{\leq R+1}$
    \begin{eqnarray*}
       \langle  \sum_u  e_u \otimes \ell^*(g_u) \xi , \sum_t e_t \otimes \zeta_t 
       \rangle & =&  \sum_u \langle   \xi ,   \ell(g_u) \zeta_u \rangle  \\ & = & \langle \xi, \mathscr{L}_{R}(\sum_u e_u \otimes \zeta_u)\rangle.  
    \end{eqnarray*} Since $\ell^*(g_u)\xi \in \mathscr{F}_q^{\leq R}$ and $\sum e_t\otimes \zeta_t \in \mathbb{C}^s \otimes \mathscr{F}_q^{\leq R}$ was arbitrary, this implies that
    \begin{equation*}
    \mathscr{L}_R^* (\xi) = \sum_u e_u \otimes \ell^*(g_u) \xi
    \end{equation*} which is precisely \eqref{eq:scriptEllStar}.
    
    We will prove that $\Vert \mathscr{L}_R\Vert \leq c_R $ by induction on $R$.  

    If $R=0$ and $\lambda_1,\dots,\lambda_s\in \mathbb{C}$, we have $$\mathscr{L}_0 \left(\sum_{r=1}^s e_r \otimes \lambda_r\Omega \right) = \sum_{r=1}^s \lambda_r g_r$$  which, by orthonormality of $g_1,\dots,g_s$, has the same norm as $\sum_{r=1}^s  e_r\otimes \lambda_r \Omega $.  Thus $\Vert \mathscr{L}_0 \Vert = 1 =c_0$. 

    Assume now that the norm inequality is true for $R$, i.e., $\Vert \mathscr{L}_R \Vert \leq c_R$, regardless of the value of $s$ and the choice of orthonormal basis $e_1,\dots,e_s$ and orthonormal family $g_1,\dots,g_s$.    
    
    Let $\xi_1,\dots,\xi_s\in \mathscr{F}_q^{\leq R+1}$.  
    Using the $q$-commutation relations \eqref{eq:qCommutation}, we compute:
    \begin{eqnarray*}
            \left\Vert \mathscr{L}_{R+1} \left(\sum_{r=1}^s e_r \otimes \xi_r \right) \right \Vert^2 
            &=& \sum_{u,v=1}^s \langle \ell(g_u) \xi_u, \ell( g_v ) \xi_v \rangle\\ 
            &=& \sum_{u,v=1}^s \langle \xi_u, \ell^*(g_u)\ell(g_v) \xi_v \rangle \\
            &=& \sum_{u=1}^s \langle \xi_u,\xi_u\rangle + q \sum_{u,v=1}^s 
                \langle \ell(g_v)^* \xi_u , \ell(g_u)^*\xi_v \rangle \\ 
                &=& \Vert \sum e_r \otimes \xi_r\Vert^2 + q B
    \end{eqnarray*}
    where we put $B = \sum_{u,v=1}^s 
                \langle \ell(g_v)^* \xi_u , \ell(g_u)^*\xi_v \rangle$.  

                Next,
    \begin{eqnarray*}
     |B| &=& \left | \left \langle  \sum_{u,v} e_u \otimes e_v \otimes \ell(g_u)^*\xi_v , 
                        \sum_{u,v} e_{v}\otimes e_u \otimes \ell(g_u)^* \xi_v 
                \right \rangle \right | \\
                &=& \left | \left \langle (U\otimes I) (\sum_{u,v} e_v \otimes e_u \otimes \ell(g_u)^*\xi_v ), 
                        \sum_{u,v} e_{v}\otimes e_u \otimes \ell(g_u)^* \xi_v 
                \right \rangle \right | 
    \end{eqnarray*} where we put $U : \mathbb{C}^s \otimes \mathbb{C}^s \to \mathbb{C}^s \otimes \mathbb{C}^s $ the isometric map $$U(e_u\otimes e_v ) = e_v \otimes e_u$$.  Since $\Vert U\Vert=1$,
    \begin{eqnarray*}
     |B| &\leq & \left | \left \langle  \sum_{u,v} e_u \otimes e_v \otimes \ell(g_u)^*\xi_v , 
                        \sum_{u,v} e_{u}\otimes e_v \otimes \ell(g_u)^* \xi_v 
                \right \rangle \right |  \\ 
                &=& \sum_{v} \Bigl\Vert \sum_u e_u\otimes \ell(g_u)^* \xi_v \Bigr\Vert^2
                \\ & = & \sum_v \Bigl \Vert \mathscr{L}_R^* (\xi_v) \Bigr\Vert^2, 
    \end{eqnarray*}
    where in the last step we applied \eqref{eq:scriptEllStar}. 
    
    Thus \begin{eqnarray*}
        |B|&\leq& \sum_v  \Vert \mathscr{L}_R\Vert^2 \Vert \xi_v\Vert^2 
        \leq c_R^2 \Bigl\Vert \sum_{v}  e_v \otimes  \xi_v \Bigr\Vert^2. 
    \end{eqnarray*}

    Putting everything together gives us:
     \begin{eqnarray*}
            \left\Vert \mathscr{L}_{R+1} \left(\sum_{r=1}^s e_r \otimes \xi_r \right) \right \Vert^2 
                &\leq & \Bigl\Vert \sum e_r \otimes \xi_r\Bigr\Vert^2 + |q| |B| \\
                &\leq& \Bigl\Vert \sum e_r \otimes \xi_r\Bigr\Vert^2 + |q| c_R^2 \Bigl\Vert \sum_{v}  e_v \otimes  \xi_v \Bigr\Vert^2 \\ 
                &\leq& (1+|q| c_R^2) \Bigl\Vert \sum_{v}  e_v \otimes  \xi_v \Bigr\Vert^2.
    \end{eqnarray*}
    Since $c_R^2 = 1 + \cdots + |q|^R$, $1+|q|c_R^2 = 1+ \cdots + |q|^{R+1} = c_{R+1}^2$ and we conclude that $$\Vert \mathscr{L}_{R+1}\Vert \leq c_{R+1},$$ completing the inductive step.

\end{proof}

\begin{lemma} \label{lem:estimateOnSum} With the notation of \S\ref{subseq:qMatrices}, for any $N$ and any $\alpha$, $\alpha'$, $k$, $k'$, $$
\left\Vert \sum_{u,v,p=1}^N e_{vp} \otimes \ell(g_{up}^{(\alpha',k')}) 
    \ell(g_{uv}^{(\alpha,k)})^*  \right\Vert \leq \frac{1}{1-|q|}.$$
\end{lemma}
\begin{proof}
    Let $f_1,\dots,f_N$ be an orthonormal basis for $\mathbb{C}^N$, and let $\xi\in \mathscr{F}_q$.

    Consider the map 
    $
    D: \mathbb{C}^N \otimes \mathscr{F}_q \to \mathbb{C}^N \otimes \mathbb{C}^N \otimes \mathbb{C}^N \otimes \mathscr{F}_q
    $
    given by $$
    D :  f_r \otimes \xi \mapsto \sum_{u,v} f_r \otimes f_u \otimes f_v  \otimes \ell(g_{uv}^{(\alpha,k)})^*\xi.$$
    Note that $D = 1\otimes \mathscr{L}^*$ where $\mathscr{L}$ is the operator associated in Lemma~\ref{lem:operatorL} to the orthonormal basis $f_u\otimes f_v$, $1\leq u,v\leq N$ for $\mathbb{C}^{N^2}$ and the corresponding orthonormal family $g_{uv}^{(\alpha,k)}$, $1\leq u,v\leq N$ in $H_\mathbb{C}$. It follows that $\Vert D\Vert \leq (1-|q|)^{-1/2}$.

    Consider also $C : \mathbb{C}^N\otimes \mathbb{C}^N \otimes \mathbb{C}^N \otimes \mathscr{F}_q\to \mathbb{C}^N \otimes \mathscr{F}_q$ given by $$
    C: f_r \otimes f_u\otimes f_v  \otimes \xi \mapsto f_r \otimes \ell(g_{u,v}^{(\alpha',k')}) \xi.$$  Once again, $C = 1\otimes \mathscr{L}$, where $\mathscr{L}$ is the operator associated in Lemma~\ref{lem:operatorL} to the orthonormal basis $f_u\otimes f_v$ and orthonormal family $g_{uv}^{(\alpha',k')}$, and so $\Vert C\Vert \leq (1-|q|)^{-1/2}$. 

    Let now $\eta = f_r \otimes \xi \in \mathbb{C}^N\otimes \mathscr{F}_q$.  Let $W$ be the unitary defined by $W(f_r\otimes f_u\otimes f_v\otimes \xi)  = (f_v \otimes f_u \otimes f_r \otimes \xi)$. Then 
    \begin{eqnarray*}
    C W D \eta &=&  C \sum_{u,v} f_v \otimes f_u \otimes f_r \otimes \ell(g_{uv}^{(\alpha,k)})^* \xi \\
    &=& \sum_{u,v} f_v \otimes \ell(g_{ur}^{(\alpha',k')}) 
    \ell(g_{uv}^{(\alpha,k)})^* \xi  \\
    &=& \sum_{u,v,p} \bigl(e_{vp} \otimes \ell(g_{up}^{(\alpha',k')}) 
    \ell(g_{uv}^{(\alpha,k)})^*\bigr) \cdot (f_r \otimes \xi) 
    \end{eqnarray*} where we use the action of the matrix units $e_{uv}$ on the basis $f_p$ by $e_{uv} f_p = \delta_{vp} f_u$. 
    Thus 
    $$ \left\Vert \sum_{u,v,p=1}^N e_{vp} \otimes \ell(g_{up}^{(\alpha',k')}) 
    \ell(g_{uv}^{(\alpha,k)})^*  \right\Vert \leq \Vert C \Vert \Vert W\Vert \Vert D\Vert \leq 
    \frac{1}{1-|q|},
    $$
    as claimed.
\end{proof}

\subsection{Strong Convergence for $-1<q<1$.}
\begin{theorem} \label{thrm:qEstimate}
Let $L_N^{(\alpha,k)}$ be as in \S\ref{subseq:qMatrices}.  Then for any $\alpha,\alpha'\in \{1,2\}$ and any  $1\leq k,k'\leq M$, 
$$ \left\Vert (L_N^{(\alpha,k)})^* L_N^{(\alpha',k')} - \delta_{\alpha\alpha'}\delta_{kk'} 1\right\Vert \leq \frac{1}{N}\cdot \frac{|q|}{1-|q|}.$$
\end{theorem}
\begin{proof}
    By definition, $$(L_N^{(\alpha,k)})^* = \frac{1}{\sqrt{N}} \sum_{u,v=1}^N e_{vu} \otimes \ell(g_{uv}^{(\alpha,k)})^*,\quad L_N^{(\alpha',k')} = \frac{1}{\sqrt{N}} \sum_{u',p=1}^N e_{u'p} \otimes \ell(g_{u'p}^{(\alpha',k')}).$$
    Thus 
    \begin{eqnarray*}
        (L_N^{(\alpha,k)})^* L_N^{(\alpha',k')} &=& 
        \frac{1}{N} \sum_{u,v=1}^N \sum_{u',p=1}^N e_{vu} e_{u'p}  \otimes \ell(g_{uv}^{(\alpha,k)})^* \ell(g_{u'p}^{(\alpha',k')} ) \\ 
        &=& \frac{1}{N} \sum_{u,v,p=1}^N e_{vp} \otimes \ell(g_{uv}^{(\alpha,k)})^* \ell(g_{up}^{(\alpha',k')} ).
    \end{eqnarray*}
    We now apply the $q$-commutation relations \eqref{eq:qCommutation} to rewrite this sum as:
     \begin{eqnarray*}
        (L_N^{(\alpha,k)})^* L_N^{(\alpha',k')} &=& 
         \delta_{kk'}\delta_{\alpha\alpha'}\sum_{v,p} e_{vp} \delta_{pv}
        + 
        \frac{q}{N} \sum_{u,v,p} e_{vp} \ell(g_{up}^{(\alpha',k')} )\ell(g_{uv}^{(\alpha,k)})^*  \\
        &=& \delta_{\alpha\alpha'}  \delta_{kk'}1 +\frac{q}{N} \Sigma
    \end{eqnarray*}
    where $\Sigma = \sum_{u,v,p} e_{vp} \ell(g_{up}^{(\alpha',k')} )\ell(g_{uv}^{(\alpha,k)})^*$.  By Lemma~\ref{lem:estimateOnSum}, $\Vert\Sigma \Vert \leq (1-|q|)^{-1}$, which gives the statement of the Theorem. 
\end{proof}

\begin{theorem}
    The family $L_N^{(\alpha,k)}$ strongly converges to a family $L^{(\alpha,k)}$ of free creation operators: for any $*$-polynomial $P$ in $2M$ variables, 
    $$
     \Vert P(L_N^{(\alpha,k)} : \alpha=1,2,\,  1\leq k\leq M)\Vert  
    \xrightarrow[N \to \infty]{} \Vert P(L^{(\alpha,k)}: \alpha=1,2,\, 1\leq k\leq M)\Vert$$
\end{theorem}
\begin{proof}
      Let $\mathscr{E}_N$ be the $C^*$-algebra generated by $L_N^{(\alpha,k)} : \alpha=1,2, 1\leq k\leq M$ and let $\mathscr{E}$ be the extended Cuntz algebra  generated by free creation operators. Fix a free ultrafilter $\omega$ and consider the ultraproduct
      $$\mathfrak{A} = {\prod}^\omega \mathscr{E}_N.$$ Theorem~\ref{thrm:qEstimate} implies that $\Vert L_N^{(\alpha,k)}\Vert^2 = \Vert (L_N^{(\alpha,k)})^* L_N^{(\alpha,k)}\Vert \leq 1 + |q|/(1-|q|) $ uniformly in $N$.
      Thus we can let $\hat{L}^{(\alpha,k)}$ be the element of the ultrapower represented by the sequence $(L_N^{(\alpha,k)})_N$, and note that by Theorem~\ref{thrm:qEstimate}, these elements satisfy the Cuntz relations $$ (\hat{L}^{(\alpha,k)} )^* \hat{L}^{(\alpha',k')} = \delta_{\alpha\alpha'}\delta_{kk'} 1. $$ By universality of $\mathscr{E}$ (cf. \cite{Cuntz,Evans:OnOn}) there is a $*$-homomorphism $\beta:\mathscr{E}\to \mathfrak{A}$ that sends $L^{(\alpha,k)}$ to $\hat{L}^{(\alpha,k)}$. 
      
      Note that the state  $\psi_N = \frac{1}{N}\operatorname{Tr} \otimes \langle \Omega,\cdot\Omega\rangle $ defines a state on each $\mathscr{E}_N$ and thus a state $\psi_\omega $ on the ultrapower $\mathfrak{A}$. 
      By \cite{Shlyakhtenko:Matrices}, the $*$-distribution of the family $L_N^{(\alpha,k)}$ converges to that of $L^{(\alpha,k)}$, 
      so that the restriction of the state $\psi_\omega$ to the $C^*$-algebra $\hat{\mathscr{E}}$ 
      generated by the family $\hat{L}^{(\alpha,k)}$ is exactly the vacuum expectation state $\psi$ on the extended Cuntz algebra.  It follows that the composition of  the GNS representation of $\hat{\mathscr{E}}$ associated to the restriction of $\psi_\omega$ with $\beta$ is exactly the GNS representation of $\mathscr{E}$ for the vacuum state $\psi$. Since this representation is faithful, $\beta$ must be an isomorphism onto $\hat{\mathscr{E}}$ and thus preserves norms.  Since this result is independent of the choice of $\omega$, the statement of the theorem follows. 
      \end{proof}

      \begin{theorem}
          Let $-1<q<1$ and let  $A_N^{(k,q)}$, $B_N^{(k,q)}$, $Q_N^{(k,q)}$, $k=1,\dots,M$ be as in \S\ref{subseq:qMatrices}.  Let $A^{(k)}, B^{(k)}$, $k=1,\dots,M$ be a free semicircular family, and let $Q^{(k)}$ be the circular elements $A^{(k)} + i B^{(k)}$.  Then the family $(A_N^{(k,q)}, B_N^{(k,q)}: k=1,\dots,M)$ converges strongly to $(A^{(k)}, B^{(k)} : k=1,\dots,M)$, and $(Q^{(k,q)}_N:k=1,\dots,M)$ converges strongly to $(Q^{(k)}:k=1,\dots,M)$.
    \end{theorem}
    \begin{proof}
        This is immediate from the strong convergence of the matrices $L_{N}^{(\alpha,k)}$, the definitions of $A_N^{(k,q)}$, $B_N^{(k,q)}$, 
        $Q_N^{(k,q)}$, and the fact that if $L^{(\alpha,k)}$ are extended Cuntz 
        isometries, then $Q^{(k)} = \sqrt{2}(L^{(1,k)} + (L^{(2,k)})^*)$ are free circular 
        elements, and thus $A^{(k)} = \frac{1}{2} (Q^{(k)} + (Q^{(k)})^*)$ and 
        $B^{(k)} = \frac{1}{2i} (Q^{(k)} - (Q^{(k)})^*)$ form a free semicircular family (see \cite{Voiculescu:Symmetries,Voiculescu:Circular}).
    \end{proof}
    
\subsection{Remarks on the case $q=-1$.} 
Consider the anti-symmetric Fock space associated to a finite-dimensional complex Hilbert space $\mathbb{C}^s$ with orthonormal basis $e_1,\dots,e_s$. Then for any $1\leq r\leq s$ there is an automorphism $\alpha$ of the CAR algebra that takes $a(e_k)$ to $a(e_k)$ if $1\leq k\leq r$, and $a(e_k)$ to $a(e_k)^*$ if $r< k\leq s$, called the ``particle-hole duality''.  To see this, let $\sigma_j = a(e_j)+a(e_j)^*$; then CAR implies that $\sigma_j$ is a self-adjoint unitary.  Let $U=\sigma_{r+1}\cdots \sigma_{s}$. It turns out that $U a(e_k) U^* = a((-1)^{s-r}e_k)$ if $1\leq k\leq r$ and $U a(e_k)U^* = a((-1)^{s-r+1}e_k)^*$ if $r<k\leq s$.  Thus composing this  automorphism with an appropriate Bogoljubov automorphism that changes the signs of $e_k$'s gives the desired map.  

In particular, this means that there exists an automorphism that changes $\ell (g_{ij}^{(1,k)})$ to $\ell(g_{ij}^{(1,k)})$ and $\ell(g_{ij}^{(2,k)})^*$ to $\ell(g_{ij}^{(2,k)})$ (note that in our notation $\ell(\cdot)=a(\cdot)^*$ when $q=-1$).  Such an automorphism transforms $Q_N^{(k,-1)} = \sqrt{2}(L_N^{(1,k)} + (L_N^{(2,k)})^*)$ to the sum $\sqrt{2}(L_N^{(1,k)} + L_N^{(2,k)}) = N^{-1/2}\sqrt{2} \sum_{uv} e_{uv}\otimes \ell( g_{uv}^{(1,k)} + g_{uv}^{(2,k)}) = {2}  N^{-1/2} \sum_{uv} e_{uv} \otimes \ell(\hat{g}_{uv}^{(k)})$ where $\hat{g}_{uv}^{(k)}$ is some other orthonormal family.  It follows that 
$$\Vert Q_N^{(k,-1)}\Vert = {2} \Vert L_{N}^{(1,k)} \Vert.$$  
On the other hand, $Q_N^{(k)}$ converges in $*$-distribution to a circular variable defined as $A^{(k)} + i B^{(k)}$ where $A^{(k)}$ and $B^{(k)}$ are semicircular variables of norm $2$, and so $\liminf_N \Vert Q_N^{(k,-1)}\Vert \geq 2\sqrt{2}$, and thus $\liminf_N \Vert L_N^{(1,k)}\Vert \geq \sqrt{2}.$  By \cite{Shlyakhtenko:Matrices}, $L_N^{(1,k)}$ converges to an extended Cuntz isometry in $*$-moments, but that isometry has operator norm $1 \neq \sqrt{2}$. Thus this $*$-convergence cannot be upgraded to strong convergence.

\end{document}